\documentclass[12pt,reqno]{amsart}
\usepackage{geometry}
\usepackage{amssymb}
\usepackage{amsmath}
\usepackage{amsfonts}
\usepackage{amsthm}
\usepackage{mathrsfs}
\usepackage{graphicx}
\usepackage[table]{xcolor}
\usepackage{xcolor}
\usepackage{enumitem}
\usepackage{url}
\usepackage{tabularx}
\usepackage{hyperref}
\usepackage{tikz-cd}

\setenumerate{label={$($\normalfont\arabic*)}}

\DeclareMathOperator{\M}{M}

\def \QQ{\mathbb{Q}}

\def \ZZ{\mathbb{Z}}

\def \RR{\mathbb{R}}

\def \CC{\mathbb{C}}
\def\<#1>{{\left\langle{#1}\right\rangle}}

\theoremstyle{plain}
\newtheorem{thm}{Theorem}[section]
\newtheorem*{thm*}{Theorem}
\newtheorem{prop}[thm]{Proposition}

\newtheorem{lemma}[thm]{Lemma}

\theoremstyle{remark}
\newtheorem{remark}[thm]{Remark}

\newtheorem*{rem*}{Remark}
\theoremstyle{definition}
\newtheorem{obs/res}[thm]{Observation}

\begin{document}
\title{An unexpected trace relation of CM points}

\author{Daniel Kohen}

\address{Universit\"at Duisburg-Essen, Fakult\"at f\"ur Mathematik}
\email{daniel.kohen@uni-due.de}
\thanks{DK was supported by an Alexander von Humboldt postdoctoral fellowship}


\begin{abstract}
Let $E/\QQ$ be an elliptic curve of conductor $N=p^2M$ where $p$ is an odd prime not dividing $M$. Let $\mathcal{O}_f$ be the order of conductor $f$  (relatively prime to $N$) in an imaginary quadratic field $K$ in which
$p$ is inert and such that the sign of the functional equation of $E/K$ is $-1$. Associated to these data there is a Shimura curve of non-split Cartan level at $p$ and a CM point of conductor $f$ on it. We can also consider a CM point of conductor $pf$ on another Shimura curve,
using a split Cartan level at $p$. These curves admit  parametrizations to $E$ and taking the images of the CM points we obtain points on $E$ defined over $H_f$ and $H_{pf}$ respectively (the ring class fields of conductor $f$ and $pf$). We prove that
these points arising from different Shimura curves satisfy a trace compatibility that is non-trivial if and only if the local sign of $E/\QQ$ at $p$ is $+1$.

\end{abstract}				
\maketitle
		
\section*{Introduction}

Let $E/\QQ$ be an elliptic curve of conductor $p^2$. Since $E$ is modular, there exists a parametrization from the classical modular curve $X_0(p^2)$ to $E$.  Let $K$ be an imaginary quadratic field  
and let $\mathcal{O}_K$ be its ring of integers. If the prime $p$ is split in $\mathcal{O}_K$, the classical theory of Heegner points \cite{MR803364} supplies a point on $X_0(p^2)$ such that its image under the modular parametrization
is $P_1 \in E(H)$, where $H$ is the maximal unramified abelian extension of $K$. Consider $P_K=\text{Tr}_{K}^{H} P_1 \in E(K)$. The celebrated Gross-Zagier formula \cite{MR833192} shows that this point is non-torsion precisely when $L^{'}(E/K,1) \neq 0$.
When  $p$ is inert in $K$ we do not have Heegner points in the modular curve $X_0(p^2)$, however, the non-split Cartan curve  $X_{ns}(p)$  hosts Heegner points and uniformizes $E$. This, together with a generalization
of the Gross-Zagier formula, was carried out by Zhang \cite{MR1868935}. 

In spite of $X_0(p^2)$ not possessing Heegner points of conductor $1$ whenever $p$ is inert, it does have CM points of conductor $p$ defined over $H_p$, the ring class field associated to the order of conductor $p$ inside $\mathcal{O}_K$. Take any such point  $P_{p} \in E(H_p)$. One question arises naturally: is $\text{Tr}_{H}^{H_p} P_p \in E(H)$ non-torsion? 

\bigskip

\begin{thm*} (Special case of Proposition \ref{prop:signo-} and Theorem \ref{prop:main2})

 Let $E/\QQ$ be an elliptic curve of conductor $p^2$ and let $K$ be an imaginary quadratic field such that $p$ is inert in $K$. If the local Atkin-Lehner sign $w_p(E)$ of $E$ at $p$ is $-1$ then $\text{Tr}_{H}^{H_p} P_p$ is zero.
If $w_p(E)=+1$, then $\text{Tr}_{H}^{H_p} P_p \in E(H)$ is non-torsion if and only if $P_{1} \in E(H)$ is non-torsion.
\end{thm*}
The theme of Heegner points constructions for more general orders (including non-split Cartan orders) was also studied in \cite{MR3766864}, \cite{MR3484373}, \cite{KP} and \cite{MR3815284}. 
An immediate application  of the present article is that, if $w_p(E)=+1$, we can explicitly compute Heegner points associated to imaginary quadratic fields in which $p$ is inert in a  much more efficient way than \cite{MR3484373},
 as we can avoid working with a non-split Cartan curve and computing the explicit modular parametrization to the elliptic curve $E$.

One of the main features of Heegner points (or CM points in general) is that they come in families satisfying trace compatibilities (see for example \cite[Section 2.4]{B-D1},\cite[Section 6]{Cornut-Vatsal},  \cite[Proposition 3.10]{MR2020572},\cite[Proposition 3.7]{MR1110395},\cite[Proposition 4.8]{MR2392363}).
These compatibilities are a crucial tool in order to bound Selmer groups  and construct $p$-adic $L$-functions. 

Cornut and Vatsal introduced the notion of \emph{good} CM points which are essentially the CM points that satisfy such compatibility relations. Our CM points of conductor $p$ are not \emph{good} (they are of type III in the sense \cite[Section 6.4]{Cornut-Vatsal}), thus the main interest of this paper is that
it provides an (unexpected!) trace relation of  CM points  living in two distinct modular curves. A parallel of this situation is the work of Bertolini and Darmon \cite{B-D1} that relies
in understanding the relation between CM points in distinct  curves. Their situation is deeper and more delicate since there is a sign change phenomenon and they have to study CM points living in Shimura curves with different ramification sets.

The modular curve $X_0(p^2)$ is isomorphic to $X_s(p)$, the modular curve associated to a split Cartan group modulo $p$. The new part of its Jacobian is isogenous over $\QQ$ (in a manner compatible with the Hecke operators) to the Jacobian of the curve $X_{ns}(p)$, providing a modular parametrization from 
$X_{ns}(p)$ to $E$. This was first proved by
Chen \cite{Chen1} comparing the traces of the Hecke operators for these curves. Afterwards, de Smit and Edixhoven \cite{MR1764312} gave another (more geometric) argument using the representation theory of $\text{GL}_2(\mathbb{F}_p)$. The result was later reproved in the aforementioned paper of Zhang 
resorting to the theory of automorphic representations. We would also like to mention the related nice article  \cite{MR3749283} which was kindly brought to our attention by Benedict Gross. The paper focus on the action of $\text{SL}_2(\mathbb{F}_p)/{\pm 1}$ on the regular differentials on $X(p)$ and its relation
with Hecke's subspace associated to forms with complex multiplication by $\QQ(\sqrt{(\frac{-1}{p})p})$.

If $w_p(E)=-1$ the proof that the trace of the point of conductor $p$ is $0$ (Proposition \ref{prop:signo-}) follows easily by studying the interplay between the Galois action and the Atkin-Lehner involution.

Suppose now that $w_p(E)=+1$. In that case, we can work with the normalizers of the Cartan groups (both split and non-split), obtaining the isogenous Jacobians $J_{ns+}(p)$ and $J^{\text{new}}_{s+}(p)$. Chen \cite{Chen2} was able to provide an explicit isogeny between these two Jacobians
(the other proofs mentioned where non-constructive). The isogeny is given in very simple terms as a double coset operator, that can be thought as a trace between the modular curves. The isogeny was also illustrated, using a nice combinatorial description of the moduli interpretation of the non-split Cartan curve,
by Rebolledo and Wuthrich \cite{MR3784056}. In the recent paper \cite{Chen3} Chen and Salari Sharif also gave a simpler proof and even showed an explicit isogeny between  $J_{ns}(p)$ and $J^{\text{new}}_{s}(p)$ (unluckily, the explicit isogeny is less nice, as it is given by a linear combination of several double coset operators).

The results of this paper boil down to prove that, when $w_{p}(E)=+1$, the  double coset operator giving Chen's isogeny applied to the CM point on the non-split Cartan curve essentially gives the trace of the CM point of conductor $p$ on the split Cartan curve (Proposition \ref{prop:main}). In addition, we prove
that the isogeny is equivariant for the Hecke operators (Proposition \ref{prop:equivariant}). Combining these results we prove the main result of the article (Theorem \ref{prop:main2}).

 Although the reader should have in mind this case for intuition, we will actually prove a more general result. Namely, we will work with an elliptic curve of conductor $N=p^2M$ where $p$
is an odd prime number that does not divide $M$,  and  with $K$  an imaginary quadratic field such that $p$ is inert in $K$, the sign of $E/K$ is $-1$ and such that if $q$ is prime and $q^2 \mid M$, then $q$ is unramified in $K$. Under these hypotheses, there exists an embedding of
$K$ into a certain indefinite quaternion algebra $B$ and we consider Shimura curves $X_{ns}$ and $X_s$ by choosing certain orders inside $B$ that, locally at $p$, look like the non-split and split Cartan orders respectively. 
For $f$  relatively prime to $N$ we consider  points $P_f \in E(H_f), P_{pf} \in  E(H_{pf})$ obtained as the images of specific special points under the parametrizations from $X_{ns}$ and $X_s$  to $E$ respectively. 
We prove that if $w_{p}(E)=-1$, then $\text{Tr}_{H_f}^{H_{pf}} P_{pf}=0$ (Proposition \ref{prop:signo-}). Conversely, if $w_{p}(E)=+1$,  we prove that the point
  $\text{Tr}_{H_f}^{H_{pf}} P_{pf} \in E(H_f)$ is non-torsion if and only if $P_f \in E(H_f)$ is non-torsion (Theorem \ref{prop:main2}).

\qquad

\textbf{Acknowledgments}: I would like to thank Ariel Pacetti and Matteo Tamiozzo for their helpful comments and discussions.

\section*{Notation}
\begin{itemize}
 \item Let $p$ be an odd prime number and let $\varepsilon$  be a non-square modulo $p$.  By $\overline{x}$ we will denote the reduction modulo $p$ of $x$, whenever it makes sense.

 \item We define the non-split Cartan order as 
 \[ M_{ns}= \left\{ \left( \begin{array}{ccc}
a & b \\
c & d  \\ \end{array} \right) \ \in \\M_{2}(\ZZ_p) : a \equiv 
d , b\varepsilon \equiv  c \bmod{p}   \right\}.\]

We also define \[M_{ns+}=M_{ns} \cup  \left\{ \left( \begin{array}{ccc}
a & b \\
c & d  \\ \end{array} \right) \ \in \\M_{2}(\ZZ_p) : a \equiv 
-d , -b\varepsilon \equiv c  \bmod{p}   \right\}. \]
 
%
\item We define the split Cartan order as \[M_{s}= \left\{ \left( \begin{array}{ccc}
a & b \\
c & d  \\ \end{array} \right) \ \in \\M_{2}(\ZZ_p) :  b \equiv  c \equiv 0 \bmod{p}   \right\},\]
and we consider \[M_{s+}=M_{s} \cup  \left\{ \left( \begin{array}{ccc}
a & b \\
c & d  \\ \end{array} \right) \ \in \\M_{2}(\ZZ_p) : a \equiv 
d \equiv 0  \bmod{p}   \right\}.\]

\item For $? \in \left\lbrace ns,ns+,s,s+ \right\rbrace$ we denote the corresponding Cartan group 

\[ C_{?}= \left\lbrace \overline{M} : M  \in M^{\times}_{?}  \right\rbrace \subset \text{GL}_{2}(\mathbb{F}_p).\]
The group $C_{ns}$ is a commutative subgroup of $\text{GL}_{2}(\mathbb{F}_p)$ isomorphic to $\mathbb{F}^{\times}_{p^2}$.
The group $C_{ns+}$ (resp. $C_{s+}$) is the normalizer of $C_{ns}$  (resp. $C_s$) and the quotients $C_{ns+}/C_{ns}, C_{s+}/C_{s}$ are both of size $2$.

\item Given a $\ZZ$-module $A$ and a prime $v$, we set $A_v=A \otimes \ZZ_v$.  Also, set $\widehat{\ZZ}= \prod_{v} \ZZ_{v}$ and $\widehat{A}= A \otimes \widehat{\ZZ}$.

\item Given a quaternion algebra $B$ (respectively an  order $R \subset B$) we denote by $B^{1}$ (resp. $R^{1}$) the elements of $B$ (resp. $R$) of reduced norm $1$.

\end{itemize}

\section{Optimal embeddings of Cartan orders}
Let $E /\QQ$ be an elliptic curve of  conductor $N=p^2 M$, where $p$ is an odd prime relatively prime to $M$. 
Let $K$ be an imaginary quadratic field such that the sign $\epsilon(E/K)$ of the functional equation of $E/K$ is equal to $-1$
and the prime $p$ is \emph{inert} in $K$. For simplicity we assume that if  $q$ is prime and $q^2 \mid M$, then $q$ is unramified in $K$.
  The sign $\epsilon(E/K)$ decomposes as the product of local signs $\epsilon_v(E/K) \in \pm 1$. Let $\eta=\prod_v \eta_v$ be the character associated to $K$ via class field theory. For each place $v$ let $\epsilon(\mathbb{B}_v)$ be the unique sign such that
\[ \epsilon_v(E/K)=\eta_v(-1)\epsilon(\mathbb{B}_v).\]

Let $S$ be the set of places such that $\epsilon(\mathbb{B}_v)=-1$. Since $\epsilon(E/K)=-1$, $S$ has odd cardinality.
 Moreover, because $K$ is imaginary, $\infty \in S$ \cite[Proposition 6.5]{MR970123}.  Let $B/\QQ$ be the unique quaternion algebra  with ramification set $S-\left\lbrace \infty \right\rbrace$.
Let $f$ be a positive integer relatively prime to $N$ and let $\mathcal{O}_f$ be the unique order of conductor $f$ inside
 $\mathcal{O}_K$, the ring of integers of $K$. 

\begin{prop}
There exists an order $R \subset B$ of discriminant $M$  and an embedding  $\iota_0: K \rightarrow B$ such that 

\[\iota_0(K) \cap R = \iota_0(\mathcal{O}_f). \]

\end{prop}

\begin{proof}
The existence of local orders and embeddings satisfying the required property is proven in \cite[Propositions 3.2, 3.4]{MR970123}.
The arguments of [Ibid. Section 9] show that we can find a global order with the required properties.
\end{proof}

\begin{remark}
Since $p^2$ divides $N$ exactly, $p \notin S$   \cite[Proposition 6.3 (2)]{MR970123}, and thus $B_{p} \simeq M_{2}(\QQ_p)$. As $R$ has discriminant relatively prime to $p$ 
we can (and do) assume that 
the local maximal order $R_p$ is $M_{2}(\ZZ_p)$. 
\end{remark}

Our next goal is to change the embedding $\iota_0$ in a fashion more suitable for our computations.

For  $? \in \left\lbrace ns,ns+,s,s+ \right\rbrace$ we define $R_{?}= \left\{ x \in R : x_p \in M_{?} \right\}$. The orders $R_{ns}$ and $R_{s}$ are both of discriminant
 $N=p^2M$. We have the following proposition.

\begin{prop} \label{prop:optimal}
There exists an embedding  $\iota: K \rightarrow B$  such that:

\begin{enumerate}
 \item The order $\mathcal{O}_{f}$   embeds optimally into $R_{ns}$, that is,
  \[\iota(K) \cap R_{ns} = \iota(\mathcal{O}_f). \]
  \item The order $\mathcal{O}_{pf}$  embeds optimally into $R_{s}$, that is,
  \[ \iota(K) \cap R_{s} = \iota(\mathcal{O}_{pf}). \]
\end{enumerate}

\end{prop}

\begin{proof}
Let $\omega_f \in \mathcal{O}_K$ be such that $\mathcal{O}_f= \ZZ + \omega_f \ZZ$.
Consider its characteristic polynomial  $\chi_{\omega_f}=X^2-tX+n$ of discriminant $D=t^2-4n$. Embedding $B$ into $B_p=M_{2}(\QQ_p)$ we get
\[ \iota_0(\omega_f)_p= \left(\begin{smallmatrix}a_0&b_0\\c_0&d_0\end{smallmatrix}\right)=A_0 \in M_{2}(\ZZ_p), \] 
where the characteristic polynomial of $A_0$ is equal to $\overline{\chi_{\omega_f}}$.
Since $p$ is inert in $K$, $D$ is not a square modulo $p$ and the reduction modulo $p$ of $A_0$  is conjugate (in $\text{GL}_{2}(\mathbb{F}_p)$) to
\[ \left(\begin{smallmatrix}  t/2&  \sqrt{D/4\varepsilon}\\ \varepsilon  \sqrt{D/4\varepsilon} & t/2\end{smallmatrix}\right) \in C_{ns},\] since both matrices have the same characteristic polynomial (with simple roots).
In addition, because  $\text{det}: C_{ns}  \rightarrow \mathbb{F}^{\times}_{p}$ is surjective and $C_{ns}$ is commutative, 
they can be conjugated by a matrix $\overline{\gamma} \in \text{SL}_{2}(\mathbb{F}_p)$.  As a consequence of strong approximation,
the map $R^{1} \rightarrow \left({R_p/ pR_p} \right) ^{1} \simeq \text{SL}_{2}(\mathbb{F}_p)$ is surjective \cite[Corollary 28.4.11]{Voight}, where the superscript $1$ denotes the elements of reduced norm $1$.
Lastly, we take a lifting $\gamma \in R^{1}$  of $\overline{\gamma}$  and we conjugate $\iota_0$ by $\gamma$ to obtain the desired $\iota$.

Statement $(1)$ is clear from the construction. For statement $(2)$, note that we only changed the order at $p$, so we just need to check locally at $p$.
 Let $\iota(\omega_f)_p=\left(\begin{smallmatrix}a&b\\c&d\end{smallmatrix}\right) \in (R_{ns})_p$. The local order $(\mathcal{O}_{pf})_p$ is equal to $\ZZ_p + pw_{f} \ZZ_p$ and thus  $\iota((\mathcal{O}_{pf})_p) \subseteq (R_s)_p$. For the other
 inclusion, note that since $p$ is inert in $K$ we must have that $p$ does not divide $bc$. Consequently, if we take an element $(\iota(x_{1}+x_{2}pw_{f}))_p \in (R_s)_p$ with $x_1,x_2 \in \QQ_p$, looking at the $(2,1)$ entry we obtain $x_{2} \in \ZZ_p$ and  looking at the $(1,1)$-entry we get $x_1 \in \ZZ_p$, as we wanted.
 \end{proof}

\section{Shimura curves}

For $? \in \left\lbrace ns,ns+,s,s+ \right\rbrace$ consider the (open)  Shimura curve whose complex points are given by the double  quotient

\[ Y_{?}(\CC)= B^{\times} \backslash (\CC-\RR) \times \widehat{B}^{\times}  / \widehat{R_{?}}^{\times}, \]

where $B^{\times}$ acts on $\CC-\RR$ via the action of $B_{\infty} \simeq GL_{2}(\mathbb{R})$ by Möbius transformations and by left multiplication on $\widehat{B}^{\times}$ and $\widehat{R_{?}}^{\times}$ acts trivially on $\CC-\RR$ and by right multiplication on $\widehat{B}^{\times}$.  For  $z \in \CC - \RR$ and $\widehat{b} \in   \widehat{B}^{\times}$  we will denote by $[z,\hat{b}]$ the class of an element in this double quotient.

We consider its compactification $X_{?}(\CC)$, given by

\[X_{?}(\CC)=Y_{?}(\CC) \cup  \left\lbrace \text{cusps} \right\rbrace. \]

The set of cusps is empty if $B \neq M_{2}(\QQ)$ and otherwise equal to the finite set
\[ \text{GL}_{2}(\QQ) \backslash \mathbb{P}^{1}(\QQ) \times \text{GL}_{2}(\widehat{\QQ}) / \widehat{R_{?}}^{\times},\]
where the action of $\text{GL}_{2}(\QQ)$ on $\mathbb{P}^{1}(\QQ)$ is given by M\"obius transformations. These curves admit the following, more classical, description.

\begin{prop}
Let $ \Gamma_{?}$ be the congruence subgroup defined as $\Gamma_{?}= \widehat{R_{?}}^{\times} \cap B^{1}$. Let $\mathcal{H^*}$ be the union of the upper half plane $\mathcal{H}$ and the set of cusps. Then the map

\[   \Gamma_{?} \backslash \mathcal{H^*}  \rightarrow X_{?}(\CC),  \, \,    [z] \mapsto [z,1], \]
is an isomorphism.

\end{prop}

\begin{proof}
The only subtle point is to check that this map is surjective. Since $B$ is an indefinite quaternion algebra, strong approximation \cite[Theorem 4.3]{Vigneras} yields (applying the determinant map) 

\[  B^{\times} \backslash \left\{ \pm 1 \right\} \times   \widehat{B}^{\times}  / \widehat{R_{?}}^{\times} \simeq {\QQ_{>0}}^{\times} \backslash  \widehat{\QQ}^{\times} /\text{det}(\widehat{R_{?}}^{\times}) \simeq  \widehat{\ZZ}^{\times}/\text{det}(\widehat{R_{?}}^{\times}). \]
In addition, as every prime $q$ such that $q^2 \mid N$ is unramified in $K$, it is easy to see that $ \text{det}(\widehat{R_{?}}^{\times})=\widehat{\ZZ}^{\times}$ and thus $ B^{\times} \backslash \left\{\pm 1 \right\} \times   \widehat{B}^{\times}  / \widehat{R_{?}}^{\times}$
consists of a single point, which immediately proves that the map is surjective.
\end{proof}

Let $J_{?}$ be the Jacobian variety of $X_{?}$. We can embed $X_{?} \hookrightarrow_{i_{?}} J_{?}$ using the so-called Hodge class \cite[Section 6.2]{MR1868935},\cite[Section 3.1.3]{MR3237437}, that is, the unique class $\xi_{?} \in Pic(X_{?}) \otimes \QQ$ of degree $1$  
such that $T_{\ell}\xi_{?}=(\ell+1)\xi_{?}$ for every prime number $\ell$ not dividing $N$. In the case $B=M_{2}(\QQ)$ the Hodge class is a (rational) linear combination of cusps. The embedding $i_{?}$ is given by sending a point $P$ to the class of $P-\xi_{?}$.

Let $\pi_{E,?}=\text{Hom}_{\xi_{?}}^{0}(X_{?},E)$, that is, the morphisms in $\text{Hom}(X_{?},E) \otimes_{\ZZ} \QQ$ sending a divisor representing $\xi_{?}$ to zero. In other words,
\[\pi_{E,?}=\text{Hom}(J_{?},E) \otimes_{\ZZ} \QQ.\]  This is the space of  parametrizations from $X_?$ to the elliptic curve $E$. Let $f_E$ be the  newform of level $N$ corresponding to $E$ by the modularity theorem.

\begin{prop}  \leavevmode

\begin{enumerate}
\item The spaces $\pi_{E,ns}$ and $\pi_{E,s}$ are $1$-dimensional.  Non-zero elements $\Phi_{ns} \in \pi_{E,ns}$ and $\Phi_s \in \pi_{E,s}$ 
 have the same eigenvalues as $f_E$ for the \emph{good} (i.e. not dividing $N$) Hecke operators and the  Atkin-Lehner involutions.  The element $\Phi_{E,s}$ factors through $J^{p-\text{new}}_{s}$, since the corresponding form is new at $p$.  

\item Furthermore, if $w_p(E)=+1$ we have the induced non-zero elements $\Phi_{ns+} \in \pi_{E,ns+}$ and  $\Phi_{s+} \in \pi_{E,s+}$ such that the following diagrams are commutative:

\begin{tikzcd}[scale=0.75]
X_{ns} \arrow[r, hook, "i_{ns}"] \arrow[d, twoheadrightarrow]
& J_{ns} \arrow[d, twoheadrightarrow]  \arrow[r, "\Phi_{ns}"] & E   \arrow[d, dash] \\
X_{ns+} \arrow[r, hook, "i_{ns+}"]
& J_{ns+} \arrow[r, "\Phi_{ns+}"] & E,
\end{tikzcd}
 \qquad 
\begin{tikzcd}[scale=0.75]
X_{s} \arrow[r, hook, "i_{s}"] \arrow[d, twoheadrightarrow]
& J^{p-\text{new}}_{s} \arrow[d, twoheadrightarrow]  \arrow[r, "\Phi_{s}"] & E   \arrow[d, dash] \\
X_{s+} \arrow[r, hook,  "i_{s+}"]
& J^{p-\text{new}}_{s+} \arrow[r, "\Phi_{s+}"] & E .
\end{tikzcd}

\end{enumerate}

\end{prop}

\begin{proof}
This is proved by \cite[Theorem 1.3.1]{MR1868935} (see also \cite[Proposition 2.6]{Gross-Prasad} and \cite[Propositions 3.7,3.8]{Cai-Shu-Tian}).

\end{proof}

\section{Explicit Galois action on special points}
The embedding $\iota: K \rightarrow B$ induces an action of $\widehat{K}^{\times}/K^{\times} \simeq \text{Gal}(K^{\text{ab}}/K)$ on the curves $X_{?}(\CC)$ by the formula

\[ {\mathfrak{a}} \cdot [z,\widehat{b}]:=[z,\hat{\iota}(\mathfrak{a}) \widehat{b}].  \]

Recall that $\omega_f \in \mathcal{O}_K$ satisfies $\mathcal{O}_f= \ZZ + \omega_f \ZZ$.  Let $z_0$ be the unique fixed point in $\mathcal{H}$ of $\iota_{\infty}(\omega_f) \in \text{GL}_{2}(\RR)$ under the action given by  M\"obius transformations. Let  $H_{f}$ (resp. $H_{pf}$) denote the ring class field of conductor $f$ (resp. $pf$), characterized by the fact that $\text{Gal}(H_{f}/K)$ (resp. $\text{Gal}(H_{pf}/K)$)
is identified with $\widehat{K}^{\times}/K^{\times}{\widehat{\mathcal{O}_{f}}}^{\times}$ (resp.
$\widehat{K}^{\times}/K^{\times}{\widehat{\mathcal{O}_{pf}}}^{\times}$).

\begin{prop}  \leavevmode
\begin{itemize}
\item The point $z_0 \in \mathcal{H}$ gives rise to $Q_f=[z_0,1]  \in X_{ns}(H_f)$.

\item The point $z_0 \in \mathcal{H}$ gives rise to $Q_{pf}=[z_0,1] \in X_{s}(H_{pf})$.

\end{itemize}
\end{prop}

\begin{proof}
Proposition \ref{prop:optimal} tells us that $Q_f \in  X_{ns}(\overline{K})$ (resp. $Q_{pf} \in  X_{s}(\overline{K})$) is fixed under the action of ${\widehat{\mathcal{O}_{f}}}^{\times}$  (resp. ${\widehat{\mathcal{O}_{pf}}}^{\times}$).

\end{proof}

The main task of this section is to understand the action $\text{Gal}(H_{pf}/H_f)$ on $Q_{pf}$.  Looking at the $p$-th component we get

\[\text{Gal}(H_{pf}/H_f) \simeq K^{\times}{\widehat{\mathcal{O}_{f}}}^{\times}/ K^{\times}{\widehat{\mathcal{O}_{pf}}}^{\times} \simeq (\mathcal{O}^{\times}_{f})_p/ (\mathcal{O}^{\times}_{pf})_p.\]
 Reducing modulo $p$ we further obtain

\[(\mathcal{O}^{\times}_{f})_p/ (\mathcal{O}^{\times}_{pf})_p \simeq (\mathcal{O}_{f}/p\mathcal{O}_{f})^{\times} / \mathbb{F}^{\times}_{p} \simeq \mathbb{P}^{1}(\mathbb{F}_{p}). \]

More explicitly, this last isomorphism is given by sending the class of $x_1+ x_2\omega_f$ (with $x_i \in \ZZ_p)$ to the element $[\overline{x_1}:\overline{x_2}] \in \mathbb{P}^{1}(\mathbb{F}_{p})$. The element $[1:0]$ is the identity and  multiplication on $\mathbb{P}^{1}(\mathbb{F}_{p})$
is given by the rule
\[ [x:1][y:1]:=[xy-n: x+y+t], \]
where, as before, $X^2-tX+n$ is characteristic polynomial of  $\iota(\omega_f)_p=\left(\begin{smallmatrix}a&b\\c&d\end{smallmatrix}\right) \in (R_{ns})_p$.

\begin{lemma}\label{lem:elementary}
 The only element of order $2$ in $\mathbb{P}^{1}(\mathbb{F}_{p})$ is $[-\overline{a}:1]$, and 
 in that case the corresponding matrix $\iota(-a+ \omega_f)_p$ belongs to $M^{\times}_{s+} \setminus M^{\times}_s$.

 Conversely, if $\iota(x_1+ x_2 \omega_f)_p \in  M_{s+}$, then
$[\overline{x_1}:\overline{x_2}] \in \left\{[-\overline{a}:1], [1:0] \right\} \subset \mathbb{P}^{1}(\mathbb{F}_{p}).$

\end{lemma}

\begin{proof}
An element $[x:1]$ has order $2$ in $\mathbb{P}^{1}(\mathbb{F}_{p})$ if and only if $2x \equiv -t \bmod{p}$. Because $\left(\begin{smallmatrix}a&b\\c&d\end{smallmatrix}\right) \in (R_{ns})_p$  has trace $t$, then 
$t=a+d  \equiv 2a \bmod{p}$, as we wanted. By the same argument, $\iota(-a+ \omega_f)_p$ has both its diagonal elements congruent to $0$ modulo $p$ (and the anti-diagonal elements not congruent to $0$) so it lies in $M^{\times}_{s+} \setminus M^{\times}_s$. Conversely, suppose that $\iota(x_1+ x_2 \omega_f)_p=\left(\begin{smallmatrix}x_1+x_2a&x_2b\\x_2c&x_1+x_2d\end{smallmatrix}\right) \in   M_{s+}$,
with $x_1,x_2 \in \ZZ_p$. If $x_{2} \equiv 0 \bmod{p}$, then we  obtain the point $[1:0]$. On the other hand, if $x_{2} \not \equiv 0 \bmod{p}$,  as $b,c$ are not divisible by $p$ this implies that both $x_1+x_2a$ and $x_1+x_2d$ are divisible by $p$, and  we get  the point $[-\overline{a}:1]$, as we wanted to prove.

\end{proof}

The local  Atkin-Lehner involution at $p$ of $X_s$ is given by right multiplication by any element in $(R_{s+})^{\times}_p \setminus (R_s)^{\times}_p$. In particular we can take the element $\left(\begin{smallmatrix}0&1\\-1&0\end{smallmatrix}\right)_{p}$.


\begin{prop} \label{prop:signo-}
If $w_p(E)=-1$, then $\text{Tr}_{H_f}^{H_{pf}} \Phi_{s} (Q_{pf})=0.$
\end{prop}

\begin{proof}
Let $\mathfrak{b}_{0}$ be the element of $\text{Gal}(H_{pf}/H_f)$ corresponding to $-a+ \omega_f$. We can group the elements of  $\text{Gal}(H_{pf}/H_f)$ into pairs of the form $\left\lbrace \mathfrak{b}, \mathfrak{b}\mathfrak{b}_{0} \right\rbrace$. Consider

\[{Q_{pf}}^{\mathfrak{b}\mathfrak{b}_{0}}=[z_0, \hat{\iota}(\mathfrak{b}) \hat{\iota}(\mathfrak{b}_{0})].\]

By Lemma \ref{lem:elementary} the element $\hat{\iota}(\mathfrak{b}_{0})$ belongs to $ \widehat{R_{s+}}^{\times} \setminus  \widehat{R_{s}}^{\times}$ and thus it can be written as a product $\left(\begin{smallmatrix}0&1\\-1&0\end{smallmatrix}\right)_{p} \cdot \hat{r}_s$, where $\hat{r}_s \in \widehat{R_{s}}^{\times}$.
Therefore we get, in $X_s$,

\[{Q_{pf}}^{\mathfrak{b}\mathfrak{b}_{0}}=[z_0, \hat{\iota}(\mathfrak{b}) \left(\begin{smallmatrix}0&1\\-1&0\end{smallmatrix}\right)_{p} \cdot \hat{r}_s]= [z_0, \hat{\iota}(\mathfrak{b}) \left(\begin{smallmatrix}0&1\\-1&0\end{smallmatrix}\right)_{p}] =w_p {Q_{pf}}^{\mathfrak{b}} .\]

As $w_p(E)=-1$,

\[ \Phi_{s}(Q_{pf}^{\mathfrak{b}\mathfrak{b}_{0}})=\Phi_{s}(w_p(E) {Q_{pf}}^{\mathfrak{b}})= -\Phi_{s}({Q_{pf}}^{\mathfrak{b}}), \]
and the result follows.
\end{proof}

In view of the above proposition, from now on we will study the case where $w_p(E)=+1$. Consider the double coset operator $\Gamma_{s+} \cdot 1 \cdot \Gamma_{ns+}$. It gives rise to a correspondence between the Shimura curves $X_{ns+}$ and $X_{s+}$. The following
result  is the heart of this article.

\begin{prop} \label{prop:main}
The following identity holds on $\text{Div}(X_{s+})$:

\[ 2(\Gamma_{s+} \cdot 1 \cdot \Gamma_{ns+})Q_f = \sum_{\sigma \in \text{Gal}(H_{pf}/H_f)}  {Q_{pf}}^{\sigma}. \]
 
\end{prop}

\begin{proof}
 Take a set of representatives $\left\lbrace \mathfrak{b}_{i} \right\rbrace$ of $\text{Gal}(H_{pf}/H_f)$ in ${\widehat{\mathcal{O}_{f}}}^{\times}$. The element $\mathfrak{r}_{i}=\widehat{\iota}(\mathfrak{b}_{i})$ belongs to $\widehat{\iota}({\widehat{\mathcal{O}_{f}}}^{\times})$ and, therefore, it also belongs to
$ \widehat{R_{ns}}^{\times}$ by Proposition \ref{prop:optimal}. We want to write $\mathfrak{r}_{i}$ as a product 
 $\gamma_{i}r_{i}$  where $r_{i} \in  \widehat{R_{s+}}^{\times}$ and
 $\gamma_{i} \in  \Gamma_{ns+}$. In order to do that, we need to modify the component at $p$. Suppose that $(\mathfrak{r}_{i})_p$ has determinant $d \in {\ZZ_p}^{\times}$.
 
 \begin{itemize}
  \item If $\overline{d}=\mu^{2}$, consider $\overline{(\mathfrak{r}_{i})_p} \left(\begin{smallmatrix} \mu^{-1} & 0\\ 0 &  \mu^{-1} \end{smallmatrix}\right) \in \text{SL}_{2}(\mathbb{F}_p)$. By strong approximation, we take a lifting of this element to an element $\gamma_{i} \in R^{1}$. Since $ \left(\begin{smallmatrix} \mu^{-1} & 0\\ 0 &  \mu^{-1} \end{smallmatrix}\right) \in C_{ns+} \cap C_{s+}$,
we get the decomposition
  \[ \mathfrak{r}_{i}= \gamma_{i} \cdot ((\gamma_{i})^{-1}\mathfrak{r}_{i}), \]
that has the desired properties.
  \item If $\overline{d}=\varepsilon\mu^{2}$, consider $\overline{(\mathfrak{r}_{i})_p} \left(\begin{smallmatrix} 0 & \mu^{-1}\\ -\varepsilon \mu^{-1} & 0 \end{smallmatrix}\right) \in \text{SL}_{2}(\mathbb{F}_p)$ 
and do the same as the previous case, using that $\left(\begin{smallmatrix} 0 & \mu^{-1}\\ -\varepsilon \mu^{-1} & 0 \end{smallmatrix}\right)   \in C_{ns+} \cap C_{s+}$.  It is worth noticing that this case  shows the necessity of working with
  the normalizers, in concordance with Proposition \ref{prop:signo-}.
 \end{itemize}

By construction we get
 
 \[  {Q_{pf}}^{\mathfrak{b}_{i}}= [z_0,1]^{\mathfrak{b}_{i}}=[z_0,\mathfrak{r}_{i}]=[z_0, \gamma_{i}r_{i}]=[{\gamma_{i}}^{-1}\cdot z_0,1]  \in X_{s+}(H_{pf}), \]
 
 which corresponds to the point ${\gamma_{i}}^{-1} \cdot z_0$ in the upper half plane of the model  $\Gamma_{s+} \backslash \mathcal{H}^{*}=X_{s+}$. 
 
Note that the correspondence $\Gamma_{s+} \cdot 1 \cdot \Gamma_{ns+}$ is of degree 
\[ [\Gamma_{ns+}: \Gamma_{ns+} \cap \Gamma_{s+}]= [C_{ns+}: C_{ns+} \cap C_{s+}]= \frac{2(p+1)}{4}=\frac{(p+1)}{2}. \]

Since $\gamma_i \in \Gamma_{ns+}$, as we vary over all $p+1$ elements of  $\text{Gal}(H_{pf}/H_f)$ we obtain the corresponding coset $\Gamma_{s+} {\gamma_{i}}^{-1} \subset \Gamma_{s+} \cdot 1 \cdot \Gamma_{ns+}$. The cosets $\Gamma_{s+} {\gamma_{i}}^{-1}$  and $\Gamma_{s+} {\gamma_{j}}^{-1}$ are equal if and only if

\[ {\gamma_{i}}^{-1} \gamma_{j} \in \Gamma_{s+}. \]

By the definition of $\gamma_i$ and $\gamma_j$ we obtain

\[ {\mathfrak{r}_{i}}^{-1} \mathfrak{r}_{j}= {r_{i}}^{-1} {\gamma_{i}}^{-1} \gamma_j r_j. \]

Since $r_i, r_j \in \widehat{R_{s+}}^{\times}$ we know that 

\[ {\gamma_{i}}^{-1} \gamma_{j} \in \Gamma_{s+}  \iff {\mathfrak{r}_{i}}^{-1} \mathfrak{r}_{j} \in \widehat{R_{s+}}^{\times} .\]

Using the second part of Lemma \ref{lem:elementary} we see that ${\mathfrak{r}_{i}}^{-1} \mathfrak{r}_{j} \in \widehat{R_{s+}}^{\times}$ if and only if the element ${\mathfrak{r}_{i}}^{-1} \mathfrak{r}_{j}$
corresponds to either $ [1:0]$ or $[-\overline{a}:1]$ in $\mathbb{P}^{1}(\mathbb{F}_{p})$.

Consequently, as we vary $i$, we run through every element of
$(\Gamma_{s+} \cdot 1 \cdot \Gamma_{ns+}) Q_f$ exactly twice, as we wanted to prove.
 
\end{proof}

\section{Chen's explicit isogeny}
Chen \cite[Theorem 2]{Chen2} proved that an explicit isogeny between $J_{ns+}(p)$ and $J^{p- \text{new}}_{s+}(p)$ is given by the double coset operator $\Gamma_{s+}(p) \cdot 1 \cdot \Gamma_{ns+}(p)$.  In fact, the nature of the proof shows that it extends to an isogeny $\Gamma_{s+} \cdot 1 \cdot \Gamma_{ns+}$  between
$J_{ns+}$ and $J^{p- \text{new}}_{s+}$. 

\begin{prop} \label{prop:equivariant}
 The map $\Gamma_{s+} \cdot 1 \cdot \Gamma_{ns+}$ is Hecke equivariant for the good Hecke operators.
\end{prop}

\begin{proof}
 Let $\ell$ be a prime that does not divide $N$. Since the reduced norm gives a surjective map  $\widehat{R}^{\times}_s \rightarrow \widehat{\ZZ}^{\times}$ we choose $\beta_{\ell} \in \widehat{R}^{\times}_s$ of reduced norm $\ell$. We perform the same construction as in Proposition \ref{prop:main} but for $\beta_{\ell}$ instead of
 $\mathfrak{r}_{i}$ and we obtain the same matrix $\gamma$ (there denoted by $\gamma_{i}$). This matrix has the property that $\alpha_{\ell}:=\gamma^{-1}\beta_{\ell}$ belongs to $\widehat{R}^{\times}_{s+} \cap \widehat{R}^{\times}_{ns+}$ and has reduced norm equal to $\ell$. Thus the $\ell$-th Hecke operator in the split side
 (resp. non-split side) is given by $\Gamma_{s+} \cdot \alpha_{\ell} \cdot \Gamma_{s+}$ (resp. $\Gamma_{ns+} \cdot \alpha_{\ell} \cdot \Gamma_{ns+}$). Using \cite[Lemma 3.29 (4)]{MR0314766} it is clear that 
 
 \[(\Gamma_{s+} \cdot 1 \cdot \Gamma_{ns+}) \cdot  (\Gamma_{ns+} \cdot \alpha_{\ell} \cdot \Gamma_{ns+})= \Gamma_{s+} \cdot \alpha_{\ell} \cdot \Gamma_{ns+} = (\Gamma_{s+} \cdot \alpha_{\ell} \cdot \Gamma_{s+}) \cdot (\Gamma_{s+} \cdot 1 \cdot \Gamma_{ns+}), \]
 and thus $\Gamma_{s+} \cdot 1 \cdot \Gamma_{ns+}$ commutes with the good Hecke operators as desired.

\end{proof}

\begin{remark}
 This proof is also in concordance with Proposition \ref{prop:signo-} and with the fact that the natural map $\Gamma_{s} \cdot 1 \cdot \Gamma_{ns}$ does not always give an isogeny between $J_{ns}$ and $J^{p- \text{new}}_{s}$. We can only prove that it commutes with the Hecke operators $T_{\ell}$ with $\ell$ a
 square modulo $p$, and given a newform $f$ of level divisible by $p^2$, the space of forms with the same eigenvalues as $f$ for these Hecke operators has (usually) dimension $2$. 
\end{remark}

As the correspondence $\Gamma=2(\Gamma_{s+} \cdot 1 \cdot \Gamma_{ns+})$ is Hecke equivariant and of degree $p+1$ it sends the Hodge class $\xi_{ns+}$ to $(p+1)\xi_{s+}$. Therefore, we can 
define 

\begin{tikzcd}[scale=0.75]
\widetilde{\Phi}_{ns+}= \Phi_{s+} \circ \Gamma  \in \pi_{E,ns+},
\end{tikzcd}
\qquad
\begin{tikzcd}[scale=0.5]
J_{ns+} \arrow[r, " \Gamma"] \arrow[dr, bend right,  " \widetilde{\Phi}_{ns+}"] 
& J^{p-\text{new}}_{s+} \arrow[d, "\Phi_{s+}"]\\
& E.
\end{tikzcd}

\begin{thm} \label{prop:main2}
 Let $E/\QQ$ be an elliptic curve of conductor $N=p^2M$ where $p$ is an odd prime not dividing $M$ and such that $w_p(E)=+1$.
Let $K$ be an imaginary quadratic field such that $p$ is inert in $K$, the sign of $E/K$ is $-1$ and such that if $q$ is prime with $q^2 \mid M$, then $q$ is unramified in $K$. 
Let $f$ be relatively prime to $N$ and consider the points $P_{pf}= \Phi_{s+}(Q_{pf}) \in E(H_{pf})$, $P_{f}=  \Phi_{ns+}(Q_{f}) \in E(H_{f})$ and $\tilde{P}_{f}=  \widetilde{\Phi}_{ns+}(Q_{f}) \in E(H_{f})$.
 Then, $\text{Tr}_{H_f}^{H_{pf}} P_{pf}= \tilde{P}_{f}   \in E(H_{f})$ and $\tilde{P}_{f}$ is non-torsion if and only if ${P}_{f}$ is non-torsion.
\end{thm}

\begin{proof}
First, note that $\text{Tr}_{H_f}^{H_{pf}} P_{pf}= \tilde{P}_{f}$ follows immediately from Proposition \ref{prop:main}. The second statement follows from the fact that the elements $\Phi_{ns+}, \widetilde{\Phi}_{ns+}$ lie in the $1$-dimensional space $\pi_{E,ns+}=\text{Hom}(J_{ns+},E) \otimes_{\ZZ} \QQ$ and $\widetilde{\Phi}_{ns+}$
is non-zero since it is the composition of the non-trivial ${\Phi}_{s+}$ and the isogeny  $\Gamma$.
\end{proof}

\begin{remark}
 For $f=1$, $P_1$ is non-torsion if $L^{'}(E/K,1) \neq 0$ by Zhang's generalization of the Gross-Zagier formula \cite[Theorem 1.2.1]{MR1868935}.
\end{remark}


\bibliographystyle{alpha}
\bibliography{bibliography}
\end{document}